\documentclass[11pt, reqno]{amsart}

\usepackage{amssymb,amsfonts}
\usepackage{amsmath,amscd}
\usepackage[all,arc]{xy}
\usepackage{enumerate}
\usepackage{mathrsfs}
\usepackage[pdftex]{graphicx}
\usepackage[utf8]{inputenc}
\usepackage[T1]{fontenc}
\usepackage[usenames,dvipsnames]{color}
\usepackage[bookmarks=false]{hyperref}
\usepackage{dsfont}
\usepackage{tikz}
\usetikzlibrary{automata,positioning}
\usepackage{multirow}

\theoremstyle{plain}
\newtheorem{thm}{Theorem}[section]

\newtheorem{prop}[thm]{Proposition}
\newtheorem{lem}[thm]{Lemma}

\theoremstyle{definition}
\newtheorem{defn}[thm]{Definition}

\newtheorem{example}[thm]{Example}

\theoremstyle{remark}

\usepackage[inline]{showlabels}

\newcommand{\eqdef}{\stackrel{\scriptscriptstyle\rm def}{=}}
\usepackage{bbm}
\usepackage{pgfplots, multicol}
\usepackage{mathtools}
\usepackage{paralist}
\usepackage{todonotes}
\usepackage{enumerate}
\usepackage{graphicx}
\usepackage{tikz-cd}

\usepackage[pdf]{pstricks}
\graphicspath{ {images/} }
\usepackage{ebezier}

\newcommand{\cM}{\mathcal{M}}

\newcommand{\bR}{{\mathbb R}}

\newcommand{\bN}{{\mathbb N}}

\newcommand{\cG}{{\mathcal G}}

\def\1{1\!\!1}

\def\and{\text{ and }}

\usepackage{pst-fill}

                        \def\^{\tilde}

\def\1{1\!\!1}

\def\rv{{\rm rv}}

\newcommand{\bbN}{\mathbb{N}} 

\newcommand{\bbR}{\mathbb{R}} 

\newcommand{\bbT}{\mathbb{T}}
\newcommand{\bbZ}{\mathbb{Z}} 

\makeatletter
\let\c@equation\c@thm
\makeatother
\numberwithin{equation}{section}
\title{Entropy spectrum of  rotation classes}

\author{Yan Mary He}
\author{Christian Wolf}

\address{Yan Mary He, Department of Mathematics, University of Toronto, M5S 2E4, Canada}
\email{yanmary.he@mail.utoronto.ca}
\address{Christian Wolf, Department of Mathematics, City College of New York}
\email{cwolf@ccny.cuny.edu}
\thanks{Wolf was partially supported by  a grant from  the Simons Foundation (\#637594 to Christian Wolf)}
\address{}
\email{}

\date{\today}

\begin{document}

\begin{abstract}
	In this note we study the entropy spectrum of rotation classes for collections of finitely many continuous potentials $\varphi_1,\dots,\varphi_m:X\to \bR$ with respect to the set of invariant measures of an underlying 
	dynamical system $f:X\to X$. We show for large classes of dynamical systems and potentials  that these entropy spectra are maximal in the sense that every value between zero and the maximum is attained. We also provide criteria that imply the maximality of the ergodic entropy spectra. For $m$ being large, our results can be interpreted as a complimentary 
	result to the classical Riesz representation theorem in the dynamical context.
\end{abstract}

\maketitle

\section{Introduction}
One of the most fundamental results in analysis is the Riesz representation theorem. One way to state this theorem is that on a locally compact Hausdorff space $X$, the space of positive linear functionals on the space $C_c(X,\bR)$ of compactly supported continuous functions from $X$ to $\bR$ is isomorphic to the space of positive Borel measures, see, e.g.\cite{Ru}. If we additionally assume that $X$ is compact then 
$C_c(X,\bR)=C(X,\bR)$ is separable and any positive measure $\mu$ is completely determined by the values  $\int \varphi_1d\mu, \int \varphi_2d\mu, \int \varphi_3 d\mu, \dots,$ where the sequence $\{\varphi_k\}_{k \ge 1}$ is dense in $C(X,\bR)$ with respect to the supremums norm. Here $C(X,\bR)$ denotes the space of continuous functions from $X$ to $\bR$.

In this paper we consider  continuous dynamical systems $f:X\to X$ on a compact metric space $(X,d)$, and restrict our attention to the set $\cM$ of $f$-invariant Borel probability measures on $X$. We then consider a countable dense set $\{\varphi_k\}_{k \ge 1}$ in $C(X,\bbR)$ and study the ``size'' of the set of measures in $\cM$ whose integrals coincide on the first $m$ potentials in terms of their entropy values (which we call the entropy spectrum \footnote{Note that our terminology is different from the term entropy spectrum which is frequently used in multifractal analysis.}). Surprisingly, even for very large $m$ and a large set of dynamical systems and potentials the entropy spectra are rather large, i.e., a non-trivial  interval $[0,h_{\rm max}]$. We also construct examples where this generic behavior fails and classify the structure of the entropy spectra for general and for ergodic measures.

We now describe our results in more details. Let $f : X \to X$ be a continuous map on a compact metric space $(X,d)$. We assume that $f$ has positive entropy and that the entropy map $\mu \mapsto h_{\mu}(f)$ is upper semi-continuous. Here $h_{\mu}(f)$ denotes the measure-theoretic entropy of $f$ with respect to $\mu$. Let $\{\varphi_k\}_{k \ge 1}$ be a dense sequence in  $C(X,\bbR)$. 
For a fixed integer $m \ge 1$ we consider the $m$-dimensional continuous potential $\Phi = (\varphi_1, ..., \varphi_m) : X \to \bbR^m$. The {\it (generalized) rotation set} of $\Phi$ is given by $${\rm Rot}(\Phi) \eqdef \{ {\rm rv}(\mu) : \mu \in \mathcal{M}  \}$$
where ${\rm rv} : \mathcal{M} \to \bbR^m$ is the rotation vector function defined by $${\rm rv}(\mu) = \left(\int_X \varphi_1 d\mu, ..., \int_X \varphi_m d\mu \right).$$ 
We shall denote ${\rm Rot}(\Phi)$ by ${\rm Rot}(m)$ to emphasize the dependence on the positive integer $m$. Since $\mathcal{M}$ is convex and compact with respect to the weak$^\ast$ topology, it follows that ${\rm Rot}(m)$ is a convex and compact subset of $\bbR^m$. Given $w \in {\rm Rot}(m)$, the set $$\mathcal{M}_{\Phi}(w) = \{\mu \in \mathcal{M} : {\rm rv}(\mu) = w \}$$ is called the {\it rotation class} of $w$. 

For $w \in {\rm Rot}(m)$, the {\it localized entropy} $H(w)$ of $w$ is defined as (see \cite{Jenkinson, KW}) $$H(w) = \sup \{ h_{\mu}(f) : \mu \in \mathcal{M}_{\Phi}(w)\}.$$

Let $\mathcal{M}_E \subset \mathcal{M}$ denote the subset of ergodic  measures.
We say that $\mu$ is a {\it periodic point measure} if  $\mu\in \cM$ is supported on the (finite) orbit of a periodic point. Let $\mathcal{M}_{per}$ denote the set of periodic point measures. For $\mathcal{M'} \subset \mathcal{M}$, we write $$h(\mathcal{M'}) = \{ h_{\mu}(f) : \mu \in \mathcal{M'}\}.$$ 
Let $K$ be a subset of $\bbR^m$. We denote by ${\rm int}(K)$ the interior of $K$ and by ${\rm ri}(K)$ the relative interior of $K$ (see Section 2 for the definitions). 

Our first main theorem is the following.
\begin{thm} \label{thm_one}
Let $f:X\to X$ be a continuous map on a compact metric space $(X,d)$. Suppose that the dynamical system $(X,f)$ has positive topological entropy and that the entropy function $\mu \mapsto h_{\mu}(f)$ is upper semi-continuous. Fix an integer $m \ge 1$. Then
\begin{enumerate}
\item[(a)] Let $\{\varphi_k\}_{k \ge 1}$ be a dense sequence in  $C(X,\bbR)$. Consider the rotation set ${\rm Rot}(m)$ for the $m$-dimensional continuous potential $\Phi = (\varphi_1, ..., \varphi_m) : X \to \bbR^m$.  Suppose $\mathcal{M}_{per}$ is dense in $\mathcal{M}_E$ (or more generally suppose the potential $\Phi$ is {\emph good}, see Definition \ref{def_good}). If $w$ is in the (relative) interior of ${\rm Rot}(m)$, then $H(w)>0$ and $h(\mathcal{M}_{\Phi}(w)) = [0, H(w)]$.

\item[(b)] We have the following four possibilities for rotation vectors $w$ which do not belong to the (relative) interior of a rotation set in $\bbR^m$.
\begin{enumerate}
	\item[(i)] $h(\mathcal{M}_{\Phi}(w)) = \{H(w)\}$,
	\item[(ii)] $h(\mathcal{M}_{\Phi}(w)) = [x, H(w)]$ for some $0<x<H(w)$,
	\item[(iii)] $h(\mathcal{M}_{\Phi}(w)) = [0,H(w)]$, and
	\item[(iv)] $h(\mathcal{M}_{\Phi}(w)) = (x, H(w)]$ for some $0\leq x<H(w)$.
\end{enumerate}
\end{enumerate}
\end{thm}

If we fix an invariant measure $\mu \in \mathcal{M}$, the Riesz representation theorem states that $\mu$ is completely determined by the values of the integrals $\int \varphi_1d\mu, \int \varphi_2d\mu, \dots$, where $\{\varphi_k\}_{k \ge 1}$ is dense in $C(X,\bbR)$. For a large but fixed $m \in \bbN$, our theorem can be interpreted as a complementary result to the Riesz representation theorem since the set of measures whose integrals coincide with $\int \varphi_1d\mu,  \dots, \int \varphi_md\mu$ is large from the perspective of entropy. Indeed, part (a) of the theorem is the generic case. More precisely, for fixed $\mu\in \cM$ and $m\in \bN$, the set of potentials $\{\varphi_1, \dots, \varphi_{m} \}$ for which the rotation vector $\rv(\mu)$ lies in the interior of the rotation set ${\rm Rot}(m)$ is open and dense in $C(X, \bbR^m)$ (see Proposition \ref{prop_dense}).

It follows from the Riesz representation theorem that as $m$ tends to infinity, the rotation classes of the rotation set ${\rm Rot}(m)$ form a decreasing sequence of covers of $\mathcal{M}$ whose intersections contain a unique invariant measure. The next theorem states that the entropy spectra may still remain rather large although the rotation classes are shrinking.
\begin{thm} \label{thm_2}
Let $f$ be as in Theorem \ref{thm_one} and suppose $\mathcal{M}_{\rm per}$ is dense in $\mathcal{M}_E$. Let $\{\varphi_k\}_{k \ge 1}$ be a dense sequence in  $C(X,\bbR)$ and for a fixed integer $m \ge 1$, let ${\rm Rot}(m)$ be the rotation set for the $m$-dimensional continuous potential $\Phi = (\varphi_1, ..., \varphi_m) : X \to \bbR^m$. Suppose ${\rm Rot}(m)$ has non-empty interior. Let $w \in {\rm int}({\rm Rot}(m))$ and let $\mu$ be an ergodic measure in $\mathcal{M}_{\Phi}(w)$. Then for each $\ell \ge m+1$, there exists an open and dense subset $S_{\ell}$ of $C(X,\bbR)$ 
such that for  fixed $\ell\geq m+1$ and $\phi_{m+1}\in S_{m+1},\dots,\phi_{\ell}\in S_{\ell}$ 
the rotation vector $(w,\int \phi_{m+1} d\mu,..., \int \phi_{\ell} d\mu)$ lies in the interior of ${\rm Rot}(\ell)$. In particular, the entropy spectrum of the rotation class of $(w,\int \phi_{m+1} d\mu,..., \int \phi_{\ell} d\mu) $ is a closed interval  which contains $[0,h_\mu(f)]$. 
\end{thm}

If $w$ is in the relative interior of ${\rm Rot}(m)$, the above theorem still holds except that one needs to work with the smallest affine space that contains ${\rm Rot}(m)$.

We now turn to the {\it ergodic}  entropy spectrum of a rotation class. For $w \in {\rm Rot}(m)$, we set 
$$\mathcal{M}_{\Phi}^E(w) \eqdef  \mathcal{M}_{\Phi}(w) \cap \mathcal{M}_E.$$
\begin{thm} \label{thm_erg}
Let $(X,f)$ be a dynamical system such that $\mathcal{M}_{per}$ is dense in $\mathcal{M}$. Let $\varphi_1,...,\varphi_{m}:X \to \bR$ be H\"older continuous potentials. Assume that every H\"older continuous potential on $X$ has a unique equilibrium state. Let $w \in {\rm ri}({\rm Rot}(m))$. Then the ergodic entropy spectrum $h(\mathcal{M}_{\Phi}^E(w))$ contains the interval $(0,H(w)]$. Moreover, there exists a dense subset $\mathcal{D}$ of ${\rm ri}({\rm Rot}(m))$ such that $h(\mathcal{M}_{\Phi}^E(w))=[0,H(w)]$ for all $w \in \mathcal{D}$. 
\end{thm}

There are various classes of dynamical systems for which the uniqueness of equilibrium states for H\"older continuous potentials has been established. For example, this holds if $f$ satisfies Bowen's specification property. For further examples beyond Bowen's specification property, we refer the reader to \cite{CT1} and the references therein.

We remark that there exists a dynamical system $(X, f)$ and a potential $\varphi : X \to \bbR$ such that for every rotation vector $w \in {\rm int}({\rm Rot}(\varphi))$, the ergodic entropy spectrum is a nontrivial interval although the equilibrium states are not unique. Such an example is exhibited in \cite{Lep}. At the end of the paper, we also present an example of a dynamical system $(X,f)$ and a good potential $\Phi : X \to \bbR^m$ such that the ergodic entropy spectrum of $w \in {\rm ri}({\rm Rot}(\Phi))$ equals a nontrivial interval union a finite set of points (see Example \ref{example}).

\section{Entropy spectrum of a rotation class}
In this section we prove Theorems \ref{thm_one} and \ref{thm_2}. Along the way, we prove Proposition \ref{prop_dense} which implies that part (a) of Theorem \ref{thm_one} is the generic case. We continue to use the notation from the previous section. 

Let $K$ be a subset of $\bbR^m$. Recall that the {\em interior} of $K$, denoted by ${\rm int}(K)$, is the union of all subsets of $K$ which are open as subsets of $\bbR^m$. The {\em relative interior} of $K$, denoted by ${\rm ri}(K)$, is the interior of $K$ with respect to the topology on the smallest affine subspace of $\bbR^m$ containing $K$. 

The proof of Theorem \ref{thm_one} consists of Proposition \ref{lem_int} and Proposition \ref{prop_bdry}, one for each part of the theorem. To prove Proposition \ref{lem_int}, given a fixed positive integer $m$ and $w \in {\rm ri}({\rm Rot}(m))$, we study the range of measure-theoretic entropy in the rotation class $\mathcal{M}_{\Phi}(w)$, namely the set $h(\mathcal{M}_{\Phi}(w))$. Since $\mathcal{M}_{\Phi}(w)$ is non-empty, compact and convex, $h(\mathcal{M}_{\Phi}(w))$ is an interval containing its right end point  $H(w)$. It turns out that under certain assumptions this interval is maximal; that is, it equals the closed interval $[0,H(w)]$. In contrast, in the proof of Proposition \ref{prop_bdry}, we construct examples of rotation sets in $\bbR^m$ and rotation vectors lying on the boundary of these rotation sets whose entropy spectra are very different.

We start with the case where $w$ is a rotation vector in the (relative) interior of the rotation set ${\rm Rot}(m)$. We say that $w\in {\rm Rot}(m)$ is a periodic point rotation vector if $\mathcal{M}_{\Phi}(w)$ contains a periodic point measure.

\begin{defn} \label{def_good}
Let $f: X \to X$ be a continuous map on a compact metric space $(X,d)$ and let $\Phi : X \to \bbR^m$ be a continuous potential. We say that the potential $\Phi$ is {\it good} if every rotation vector in the relative interior of the rotation set ${\rm Rot}(\Phi)$ can be written as a convex combination of periodic point rotation vectors.

We remark that there are many classes of dynamical systems $f:X \to X$ for which every continuous potential $\Phi : X \to \bbR^m$ is good. Such classes contain systems which satisfy any one of the following properties:
\begin{enumerate}
	\item[(i)] $\mathcal{M}_{per}$ is dense in $\mathcal{M}_E$.
	\item[(ii)] $(X,f)$ satisfies Bowen's specification property.
	\item[(iii)] $(X,f)$ satisfies the gluing orbit property introduced by Bomfim and Varandas in \cite{BV}.
	\item[(iv)] $(X,f)$ satisfies the closability property introduced by Gelfert and Kwietniak in \cite{GK}.
\end{enumerate}
\end{defn}

It is not difficult to see that if $(X,f)$ satisfies property (i), then every continuous potential $\Phi: X \to \bbR^m$ is good, since the convex hull of ergodic rotation set $$\{{\rm rv}(\mu) : \mu \in \mathcal{M}_E  \}$$
equals the rotation set ${\rm Rot}(m)$ (\cite{KWtoy}). If $(X,f)$ satisfies Bowen's specification property or the gluing orbit property, then $\mathcal{M}_{per}$ is dense in $\mathcal{M}$ and therefore is dense in $\mathcal{M}_E$. If $(X,f)$ satisfies Gelfert and Kwietniak's closability condition, then by Theorem 4.11 in \cite{GK}, $\mathcal{M}_{per}$ is dense in $\mathcal{M}_E$.

We recall the definition of the gluing orbit property which is a fairly new specification-like property introduced in \cite{BV}. A continuous map $f : X \to X$ on a compact metric space $(X,d)$ satisfies the {\it gluing orbit property} if for any $\varepsilon>0$, there exists an integer $N = N(\varepsilon) \ge 1$ such that for any points $x_1, x_2, ..., x_k \in X$ and any positive integers $n_1, ..., n_k$, there are $p_1,...,p_k \le N(\varepsilon)$ and a point $x \in X$ such that $d(f^j(x), f^j(x_1)) \le \varepsilon$ for every $0 \le j \le n_1$ and 
$$d(f^{j+n_1+p_1+...+n_{i-1}+p_{i-1}}(x), f^{j}(x_i)) < \varepsilon$$ for every $2 \le i \le k$ and $0 \le j \le n_i$ (see \cite{BV}, Definition 2.1).

The closability property is a more general version of a property guaranteed by the so-called closing lemma. Roughly speaking, if $K$ is a subset of the set of periodic points of a dynamical system $(X,f)$, then $f$ is said to have the $K$-closability property if for every ergodic measure, there is a generic point $x$ such that for any $\varepsilon>0$ there is an infinite number of $n$ such that $\{x, f(x),...,f^{n-1}(x)  \}$ can be $\varepsilon$-shadowed by $y \in K$ with a period which is roughly $n$. The closability property implies that $\mathcal{M}_{per}$ is dense in $\mathcal{M}_E$ (\cite{GK} Theorem 4.11).
More details on closability can be found in \cite{GK}, Section 4.

There are many examples of dynamical systems satisfying at least one of the above conditions (i)-(iv). The Specification Theorem (\cite{KaHa} Theorem 18.3.9) states that a diffeomorphism restricted to a compact locally maximal hyperbolic set has the specification property. Moreover, mixing interval maps and mixing cocyclic shifts, in particular mixing sofic shifts and mixing subshift of finite type, have the specification property. Every system having the specification property has the gluing orbit property. 
Moreover, a $C^0$-generic system restricted to an isolated chain-recurrent class has the gluing orbit property (\cite{Lima} Corollary 4.2). It is shown in \cite{GK} that every $\beta$-shift or $S$-gap shift is closable with respect to an appropriate subset of the set of periodic points (\cite{GK} Proposition 4.9 and Proposition 4.10).

\begin{prop} \label{lem_int}
Let $f$ be as in Theorem \ref{thm_one}. Let $\{\varphi_k\}_{k \ge 1}$ be a dense sequence in  $C(X,\bbR)$. For a fixed positive integer $m$, consider the rotation set ${\rm Rot}(m)$ for the $m$-dimensional continuous potential $\Phi = (\varphi_1, ..., \varphi_m) : X \to \bbR^m$.  Suppose $\Phi$ is {\emph good}. 
If ${\rm Rot}(m)$ has non-empty interior and if $w \in {\rm int}({\rm Rot}(m))$, then $H(w)>0$ and $h(\mathcal{M}_{\Phi}(w)) = [0, H(w)]$. If ${\rm Rot}(m)$ has empty interior, then the analogous statement holds for $w \in {\rm ri}({\rm Rot}(m))$.
\end{prop}
\begin{proof}
Suppose ${\rm Rot}(m)$ has non-empty interior. Since $\Phi$ is good, any $w \in {\rm int}({\rm Rot}(m))$ can be written as a convex combination of $m+1$ periodic point rotation vectors $w_1, ..., w_{m+1}$ in ${\rm Rot}(m)$, i.e., $w = c_1w_1 + ... + c_{m+1}w_{m+1}$ for positive real numbers $c_1, ..., c_{m+1}$ with $c_1+\dots+c_{m+1}=1$. 
For each $i=1,...,m+1$, let $\mu_i^*$ be a periodic point measure in the rotation class of $w_i$.

Define $\varphi_{m+1} : X \to \bbR$ by $\varphi_{m+1}(x) = dist(x, X')$, where $X' \subset X$ is the union of the support of the $\mu_i^*, i =1,...,m+1$. That is, $X'$ is the disjoint union of the $m+1$ periodic orbits. We add $\varphi_{m+1}$ to $\varphi_1,\dots, \varphi_m$. It follows that the set $$I = \left \{ \int_X \varphi_{m+1} d\mu : \mu \in \mathcal{M}_{\Phi}(w)\right \}$$ is a compact interval $[0,b]$ for some $b>0$.

For $x \in [0,b]$, we define $H(x)$ to be 
$$H(x) =  \sup\left\{h_\mu(f) : \int_X \varphi_{m+1} d\mu =x \text{ and } \mu \in \mathcal{M}_{\Phi}(w) \right\}.$$
At the endpoint $0$ of $I$, we have $H(0) = 0$ since 
\begin{align*}
H(0) &= \sup\left\{h_\mu(f) : \int_X \varphi_{m+1} d\mu =0 \text{ and } \mu \in \mathcal{M}_{\Phi}(w) \right\}\\
& = \sup \{h_\mu(f) : \mu = c_1\mu^*_1 +...+c_{m+1}\mu^*_{m+1} \}.
\end{align*}
On the other hand, there exists $x_0 \in (0,b]$ such that $H(x_0) = H(w)$ which is the maximum value of $H$ on $[0,b]$. We observe that $H(w)>0$. This follows from the facts that the localized entropy function is concave, $H(w_0)>0$ where $w_0$ is the rotation vector of a measure of maximal entropy of $f$ and $w \in {\rm int}({\rm Rot}(m))$.

Since the entropy map $\mu \mapsto h_{\mu}(f)$ is upper semi-continuous, $H$ is continuous on $[0,b]$. Therefore  $H([0,b])$ is the closed interval $[0, H(w)]$ which implies $h(\mathcal{M}_{\Phi}(w)) = [0, H(w)]$.

If ${\rm Rot}(m)$ has empty interior, we can consider its relative interior. Applying an analogous argument as above, we see that the conclusion remains valid.
\end{proof}


In the following proposition, we will see that there are further possibilities for the entropy spectrum of rotation vectors which do not lie in the (relative) interior of a rotation set.
 
\begin{prop} \label{prop_bdry}
Fix an integer $m \ge 1$. We have the following four possibilities for the entropy spectrum of a rotation vector which does not lie in the (relative) interior of a rotation set in $\bbR^m$:
\begin{enumerate}
\item[(i)] $h(\mathcal{M}_{\Phi}(w)) = \{H(w)\}$,
\item[(ii)] $h(\mathcal{M}_{\Phi}(w)) = [x, H(w)]$ for some $0< x < H(w)$, 
\item[(iii)] $h(\mathcal{M}_{\Phi}(w)) = [0,H(w)]$, and
\item[(iv)] $h(\mathcal{M}_{\Phi}(w)) = (x, H(w)]$ for some $0\leq x<H(w)$.
\end{enumerate}
\end{prop}
Since $h(\mathcal{M}_{\Phi}(w))$ is a connected subset of $\bbR$ with maximum value $H(w)$, it is a (possibly trivial) interval with maximum value $H(w)$. The four cases in the proposition are all the possibilities of such intervals. We construct examples in the following three lemmas, one for each of the first three possibilities. However, we do not know if possibility (iv) can be realized.
In the following three lemmas $X$ is a two-sided full shift in $d$ symbols where $d \ge 2$ is an integer. 
\begin{lem}
There exist a rotation set $\mathcal{R} \subset \bbR^m$ and a rotation vector $w$ not in the (relative) interior of $\mathcal{R}$ such that and $h(\mathcal{M}_{\Phi}(w))$ is a singleton.
\end{lem}
\begin{proof}
Let $X_{\alpha} \subset X$ be a minimal and uniquely ergodic subshift of $X$ with entropy $\alpha \ge 0$. We note that for $\alpha>0$ examples of such shifts are constructed in \cite{Gr}. For $\alpha=0$ we chose $X_\alpha$ to be a periodic orbit.
Let $\varphi : X \to \bbR$ be $\varphi(x) = dist(x, X_\alpha)$. Let $\{\varphi_k\}_{k \ge 1}$ be a dense sequence in $C(X,\bbR)$.
We insert $\varphi$ into the sequence $\{\varphi_k\}_{k \ge 1}$ as $\varphi_1$ and set $\mathcal{R} = {\rm Rot}(\varphi_1,...,\varphi_{m})$. 
Consider the intersection of the hyperplane $\{x_1 = 0 \}$ and $\mathcal{R}$. 
This intersection is a singleton which we denote by $w$. Then $w \in \partial \mathcal{R}$. Note that the rotation class $\mathcal{M}_{\Phi}(w)$ of $w$ is a singleton which is the unique invariant measure on $X_\alpha$. Therefore, $h(\mathcal{M}_{\Phi}(w)) =\{ \alpha\}$.
\end{proof}

\begin{lem}\label{lem25}
There exist a rotation set $\mathcal{R} \subset \bbR^m$ and a rotation vector $w$ not in the (relative) interior of $\mathcal{R}$ such that $h(\mathcal{M}_{\Phi}(w)) = [0, H(w)]$ with $H(w)>0$.
\end{lem}
\begin{proof}
Let $Y \subset X$ be a transitive subshift of finite type of positive entropy. Consider the potential $\varphi : X \to \bbR$ given by $\varphi(x) = dist(x, Y)$.
As in the proof of the previous lemma, we insert $\varphi$ into a dense sequence $\{\varphi_k\}_{k \ge 1}$ as $\varphi_1$ and set $\mathcal{R}= {\rm Rot}(\varphi_1,...,\varphi_{m})$. Consider the intersection of the hyperplane $\{x_1 = 0 \}$ and $\mathcal{R}$. Again this intersection is a singleton and we denote it by $w$. Then $w \in \partial \mathcal{R}$. The rotation class of $w$ 
$$\mathcal{M}_{\Phi}(w) = \{ \mu \in \mathcal{M} : {\rm supp}(\mu) \subset Y  \}.$$
Therefore, $h(\mathcal{M}_{\Phi}(w)) = [0, H(w)]$. In this case, $H(w)$ equals the entropy of $(Y,f)$ which is positive.
\end{proof}
We note that in construction of Lemma \ref{lem25} we even obtain $h(\mathcal{M}_{\Phi}^E(w)) = [0, H(w)]$. This follows since for the transitive SFT $Y$ every entropy value between $0$ and $h_{\rm top}(Y)$ is attained by an ergodic measure.

\begin{lem}
There exist a rotation set $\mathcal{R} \subset \bbR^m$ and a rotation vector $w$ not in the (relative) interior of $\mathcal{R}$ such that $h(\mathcal{M}_{\Phi}(w)) = [x, H(w)]$ where $0< x < H(w)$.
\end{lem}
\begin{proof}
Let $Y \subset X$ be the union of two disjoint minimal and uniquely ergodic subshifts $Y_1$ and $Y_2$ with positive but distinct entropies. For $i=1,2$, denote by $h_i$ the entropy of $Y_i$ and assume without loss of generality that $h_1 < h_2$. Let $\varphi : X \to \bbR$ be the potential given by $\varphi(x) = dist(x, Y)$. Again we insert $\varphi$ into a dense sequence $\{\varphi_k\}_{k \ge 1}$ as $\varphi_1$ and set $\mathcal{R} = {\rm Rot}(\varphi_1,...,\varphi_{m})$. As in the previous proof, there is exactly one intersection point $w \in \bbR^m$ of the hyperplane $\{x_1 = 0 \}$ and $\mathcal{R}$. Then $w \in \partial \mathcal{R}$. The rotation class of $w$ 
$$\mathcal{M}_{\Phi}(w) = \{ t\mu_1 + (1-t)\mu_2 : t \in [0,1]  \}$$
where $\mu_i$ is the unique invariant measure on $Y_i, i=1,2$.
Therefore, we have $h(\mathcal{M}_{\Phi}(w)) = [h_1, h_2]$ where $0< h_1<h_2$.
\end{proof}

Let $f:X\to X$ be a continuous map on a compact metric space $(X,d)$ and let $\mu\in \cM$. We say $\mu$ can be approximated by periodic point measures if there exists $(\mu_k)_{k\in\bN}\subset \cM_{\rm per}$ with $\mu_k\not=\mu$  and $\mu_k\to \mu$ in the weak$^\ast$ topology as $k\to \infty$.
\begin{prop} \label{prop_dense}
Let $f:X\to X$ be a continuous map on a compact metric space. 
Suppose $\mu_0\in \cM$ is a measure which can be approximated by periodic measures. Let $m\in \bN$. Then the set
\[
CO(\mu_0)=\{\Phi\in C(X,\bR^m): \rv(\mu_0)\in {\rm int}({\rm Rot}(\Phi))\}
\]
is open and dense in $C(X, \bbR^m)$.
\end{prop}
\begin{proof}
The assertion that $CO(\mu_0)$ is open in $C(X, \bbR^m)$ follows from the fact that rotation sets are compact and convex sets which depend continuously on the potential $\Phi$ in the Hausdorff metric.

Next we assume that $\Phi \in C(X, \bbR^m)$ with $\rv(\mu_0)\not\in {\rm int}({\rm Rot}(\Phi)).$ It follows from the hypothesis that $\mu_0$ can be approximated by periodic 
point measures that $\cM_{\rm per}$  is infinite. Thus, by making an arbitrary small perturbation of the potential $\Phi$ near $m+1$ periodic orbits if necessary, we may assume that 
${\rm int}({\rm Rot}(\Phi)) \not=\emptyset.$ Hence $\rv(\mu_0)\in \partial  {\rm Rot}(\Phi)$. Fix $\epsilon>0$. We only consider the case $m=2$. The case $m>2$ can be proven analogously. Since  $\mu_0$ can be approximated by periodic  measures there are periodic points $x_1,x_2,x_3$ and associated distinct periodic point measures $\mu_1,\mu_2,\mu_3$
such that $\mu_i\not=\mu_0$  and $\Vert\rv(\mu_i)-\rv(\mu_0)\Vert<\frac{\epsilon}{3}$ for $i=1,2,3$. Therefore, there exists $v_1,v_2,v_3\in \bR^2$ with $\Vert v_i\Vert<\frac{\epsilon}{2}$ and $r>0$ such that the ball $B(\rv(\mu_0),r)$ about $\rv(\mu_0)$ with radius $r$ is contained in the interior of the  triangle with corners $\rv(\mu_i)+v_i, i=1,2,3$. Let $V_i$ be small open neighborhoods in $X$ of  the periodic orbits  $O(x_i)$. By applying Urysohn's lemma there exists $\Phi_\epsilon:X\to\bR^2$ such that $\Vert \Phi_\epsilon\Vert <\frac{\epsilon}{2}$, $\Phi_\epsilon(x)=v_i$ for all $x\in O(x_i)$ and $\Phi_{\epsilon}(x)= 0$ for all $x\in X\setminus (V_1\cup V_2\cup V_3)$. We define $\Psi=\Phi+\Phi_{\epsilon}$.
Since the Borel measure $\mu_0$ differs with the periodic point measures $\mu_i$, we can make $\mu_0(V_i)$ as small as necessary  by decreasing the seize of the neighborhoods $V_i$. It follows that if the sets $V_i$ are small enough, then $\rv_{\Psi}(\mu_0)\in  B(\rv(\mu_0),r)$. We conclude that $\Vert \Phi-\Psi\Vert <\epsilon$ and $\rv_\Psi(\mu_0)\in {\rm int}({\rm Rot}(\Psi))$. This shows that $CO(\mu_0)$ is dense in $C(X,\bR^m)$.
\end{proof}

We end this section by presenting the proof of Theorem \ref{thm_2}.

\begin{proof}[Proof of Theorem \ref{thm_2}]
We need to show that for each $\ell \ge m+1$, there exists an open and dense subset $S_{\ell}$ of $C(X,\bbR)$ which satisfies the desired property. The openness follows from the fact that the rotation set ${\rm Rot}(\ell)$ depends continuously in the Hausdorff metric on the $\ell$ potentials $(\varphi_1,...,\varphi_{\ell})$. 

We prove the denseness by induction. Let $\ell = m+1$. Choose any $\varphi_{m+1} \in C(X, \bbR)$. We are done if $(w, \int \varphi_{m+1}d\mu)$ is in the interior of ${\rm Rot}(\varphi_1,...,\varphi_{m+1})$. Otherwise suppose $(w, \int \varphi_{m+1}d\mu)$ lies on the boundary of ${\rm Rot}(\varphi_1,...,\varphi_{m+1})$. Since $\mu$ is ergodic and $\mathcal{M}_{\rm per}$ is dense in $\mathcal{M}_E$, $\mu$ can be approximated by periodic measures. By the same surgery argument as in the proof of Proposition \ref{prop_dense}, we can modify the potential $\varphi_{m+1}$ so that the modified potential $\varphi'_{m+1}$ is close to $\varphi_{m+1}$ in $C(X,\bbR)$ and $(w, \int \varphi'_{m+1}d\mu)$ becomes an interior point of ${\rm Rot}(\varphi_1,..., \varphi'_{m+1})$.

Therefore  there exists a dense subset $S_{m+1}$ in $C(X,\bbR)$ such that the point $(w, \int \varphi_{m+1}d\mu)$ is in the interior of ${\rm Rot}(m+1) = {\rm Rot}(\varphi_1,...,\varphi_{m},\varphi_{m+1})$ where $\varphi_{m+1} \in S_{m+1}$. By Theorem \ref{thm_one} part (a), the entropy spectrum of the rotation class of $(w, \int \varphi_{m+1}d\mu)$ contains $[0,h_\mu(f)]$. For $\ell \ge m+2$, the proof follows by induction using the same argument.
\end{proof}

\section{Ergodic entropy spectrum of a rotation class}
Finally, we prove Theorem \ref{thm_erg}. 

\begin{proof}[Proof of Theorem \ref{thm_erg}]
Let $w \in {\rm ri}({\rm Rot}(m))$. Since $\mathcal{M}_{per}$ is dense in $\mathcal{M}$, there exists $\mu_w$ in the rotation class of $w$ such that $$\mu_w = c_1\mu_1 + ... +c_{m+1}\mu_{m+1}$$
for some positive real numbers $c_i$ and periodic point measures $\mu_i$. Consider the potential $\varphi: X \to \bbR$ given by $\varphi(x) = {\rm dist}(x, X')$ where $X'$ is the union of the support of the $\mu_i$. We note that $\varphi$ is H\"older continuous. Insert $\varphi$ into the dense sequence $\{\varphi_k\}_{k \ge 1}$ as $\varphi_{m+1}$. Then the set of $(m+1)$-th coordinates of the rotation set ${\rm Rot}(m+1)$ $$I = \left \{ \int_X \varphi d\mu : \mu \in \mathcal{M}_{\Phi}(w)\right \}$$ is a compact interval $[0,b]$ for some $b>0$.

By a slight abuse of notation we also denote the line segment in $\bbR^{m+1}$ connecting $(w,0)$ to $(w,b)$ by $I$. We consider the localized entropy function $H$ restricted to $I$. At the endpoint $(w,0)$ of $I$ we have $H(w, 0) = 0$. 
Let $\mu_{\rm max}\in \cM_\Phi^E(w)$ with $h_{\mu_{\rm max}}(f)=H(w)$. The existence of $\mu_{\rm max}$ follows from the fact that the set of entropy maximizing measures within a 
rotation class is a compact and convex set whose extreme points are the ergodic measures.
Then $x_{\rm max}=\int \varphi d\mu_{\rm max} \in (0,b]$. It  follows from a  theorem of Jenkinson (\cite{Jenkinson} that for each $s\in (x_{\rm max},b) $ there exists a linear combination
of the potentials $\varphi_1,\dots,\varphi_{m+1}$ such that the corresponding equilibrium measure $\mu_s$ satisfies $\rv(\mu_s)=s$ and $h_{\mu_s}(f)=H(s)$. It now follows from
the fact that equilibrium measures are ergodic and from the continuity of the localized entropy function that $(0,H(w)]\subset h(\mathcal{M}_{\Phi}^E(w))$.

Since $\mathcal{M}_{per}$ is dense in $\mathcal{M}$, there exists a dense subset $\mathcal{D}$ of ${\rm ri}({\rm Rot}(m))$ such that if $w \in \mathcal{D}$, the rotation class of $w$ contains a periodic point measure whose support is denoted by $X''$. Taking $\varphi_{m+1} = dist(x,X'')$ and applying the same argument as above, we obtain $h(\mathcal{M}_{\Phi}^E(w)) = [0,H(w)]$ for any $w \in \mathcal{D}$.
\end{proof}

Finally, we present an example of a dynamical system and a good potential $\Phi$ such that the ergodic entropy spectrum of $w \in {\rm int}({\rm Rot}(\Phi))$ is the union of a nontrivial interval and a finite set of points.
Our example is inspired by Example 2 in Section 3 of \cite{Petersen}.

\begin{example} \label{example}
Let $Y$ be the  full shift on $d$ letters $\{0,...,d-1\}$ where $d \ge 2$ is an integer. Let $Y_1, Y_2 \subset Y$ be two disjoint minimal and uniquely ergodic subshifts with distinct strictly positive entropies. As noted before examples of such shifts are constructed in \cite{Gr}. 
	
In order to describe our dynamical system $(X,\sigma)$, we construct a directed labeled graph $\Gamma$.
Let $\Gamma$ be an irreducible labeled directed graph with a countable set of vertices $\{v_0, v_{11}, v_{21}, v_{12},v_{22}, \dots \}$ as shown in Figure \ref{fig_ex}. We note that in $\Gamma$, $x_{ij}$ is the unique directed edge pointing to the vertex $v_{ij}$ for all $i=1,2$ and $j \ge 1$. We assign each edge $x_{ij}$ a letter in $\{0,...,d-1\}$ in such a way that for $i=1,2$, the orbit of the point $x_{i1}x_{i2}x_{i3}\dots$ is dense in the minimal shift $Y_i$ respectively.
	
For $i=1,2$, we will construct directed edges $y_{ij}$ from $v_{ij}$ to $v_0$ for certain values of $j \ge 1$ so that the dynamical system $(X,\sigma)$  satisfies an entropy inequality specified below. We will describe the edges $y_{ij}$ later.
	
For $i = 1,2$ and $j \ge 1$, define $$g_{ij} = x_{i1}x_{i2}...x_{ij}y_{ij}$$
	if $g_{ij}$ represents a loop from $v_0$ to $v_0$ in $\Gamma$ which does not go through $v_0$ in between; otherwise $g_{ij}$ is defined to be the empty string.
	For $i = 1,2,$ let $\mathcal{G}_i = \{ g_{ij}: j \ge 1 \}$ and let $\mathcal{G}  =\mathcal{G}_1 \cup \mathcal{G}_2$. Following \cite{BDWY} we define $$X_{{\rm seq}}(\mathcal{G}_i) \eqdef \{...g_{k_{-2}}g_{k_{-1}}g_{k_{0}}g_{k_{1}}g_{k_{2}}... : g_{k_j} \in \mathcal{G}_i, j \in \bbZ\}$$
	and
	$$X_{{\rm seq}}(\mathcal{G}) \eqdef \{...g_{k_{-2}}g_{k_{-1}}g_{k_{0}}g_{k_{1}}g_{k_{2}}... : g_{k_j} \in \mathcal{G}, j \in \bbZ\}.$$
	Let $X_i$ be the topological closure of $X_{{\rm seq}}(\mathcal{G}_i)$; that is, $X_i$ is the smallest shift space that
	contains $X_{{\rm seq}}(\mathcal{G}_i)$. Similarly, let $X$ be the closure of $X_{{\rm seq}}(\mathcal{G})$.
	
	For $i =1,2$ consider the spaces $$X_{{\rm min}}(i) = X_i \setminus X_{{\rm seq}}(\mathcal{G}_i).$$ We note that $X_{{\rm min}}(i) = Y_i$. Therefore they are two disjoint minimal and uniquely ergodic subshifts with positive but different entropies in $X$.
	
	If we set 
	$$\mathcal{G}_n \eqdef \{g_{11},...,g_{1n},g_{21},...,g_{2n} \}$$ and consider the sofic shifts $X(\mathcal{G}_n)$ of $X_{\rm seq}(\mathcal{G}_n)$, then we can define the directed edges $y_{ij}$ in such a way that the topological entropies $h_{\rm top}(X(\mathcal{G}_n))$ of $X(\mathcal{G}_n)$ satisfy the following inequality
	$$\sup_n h_{\rm top}(X(\mathcal{G}_n)) < \min\{h_{\rm top}(X_{\rm min}(1))  ,h_{\rm top}(X_{\rm min}(2)) \}.$$
	Indeed, if we denote by $\mathcal{L}_m(X(\mathcal{G}_n))$ the set of words of length $m$ that appear somewhere in $X(\mathcal{G}_n)$, then the topological entropy of  $X(\mathcal{G}_n)$ is given by $$h_{\rm top}(X(\mathcal{G}_n)) = \lim_{m \to \infty} \frac{1}{m} \log \# \mathcal{L}_m (X(\mathcal{G}_n)).$$ For $i=1,2$, let $j_i^*$ denote the integer for which $y_{ij_i^*}$ is the first directed edge from $v_{ij_i^*}$ to $v_0$. For $j > j_i^*$, we construct an edge $y_{1j}$ for $j$ being an odd multiple of $j_i^*$ and $y_{2j}$ for $j$ being an even multiple of $j_i^*$.
	We can choose $j_i^*$ large enough so that
	$$h_{\rm top}(X(\mathcal{G}_n)) < \min\{h_{\rm top}(X_{\rm min}(1))  ,h_{\rm top}(X_{\rm min}(2)) \} -\epsilon$$ for some $\epsilon>0$ and all $n \ge 1$. This follows from the fact that $h_{\rm top}(X(\mathcal{G}_n)) \to 0$ as $j_i^* \to \infty$.

	Now we construct a good H\"older continuous potential $\Phi: X \to \bbR^m$ where $m$ is a fixed positive integer. Let $w_0 \in \bbR^m$ and let $w_1, w_2,..., w_{m+1} \in \bbR^m$ such that $w_0$ can be written as a convex combination of $w_1, w_2,..., w_{m+1}$. Let $P_0,P_1,...,P_{m+1}$ be periodic orbits in $X_{\rm seq}$. We define $\Phi$ as follows. Define $$\Phi |_{X_{\rm min}(1) \cup X_{\rm min}(2)} = w_0,$$ and for each $i=0,1,...,m+1$ define $$ \Phi |_{P_i} = w_i.$$ Let $V_i$ be a small neighborhood of $P_i$ such that $V_i$ does not intersect $X_{\rm min}(1) \cup X_{\rm min}(2)$. By using Urysohn's lemma to define $$\Phi |_{(X \setminus \cup_j V_j)} = w_0.$$ Moreover, we may assume that $\Phi$ is H\"older continuous. 
	
It follows that $\Phi$ is a good H\"older potential and that $w_0$ is in the interior of ${\rm Rot}(\Phi)$. Moreover, the ergodic entropy spectrum $h(\mathcal{M}^E_{\Phi}(w_0))$ of $w_0$ is contained in $$ES=[0, \alpha] \dot\cup \{h_{\rm top}(X_{\rm min}(1)), h_{\rm top}(X_{\rm min}(2))\}$$ for some $0<\alpha\leq \sup_n h_{\rm top}(X(\mathcal{G}_n))$.
Moreover, $ES\setminus \{\alpha\}\subset h(\mathcal{M}^E_{\Phi}(w_0))$.
Indeed,  that $\alpha>0$ follows from applying 
Theorem  \ref{thm_one} (a)  to $X(\cG_n)$ for $n$ sufficiently large that the periodic orbits $P_0,P_1,...,P_{m+1}$ are contained in $X(\cG_n)$.
Moreover, that every value $0\leq h<\alpha$  is attained in $h(\mathcal{M}^E_{\Phi}(w_0))$ can be shown by applying similar arguments as in the proof of Theorem \ref{thm_erg} to the sofic shifts $X(\cG_n)$. 
		
	In fact, one can generalize this construction so that the ergodic entropy spectrum of $w_0$ is contained in the disjoint union of $[0, \alpha]$ for some $\alpha>0$  and a set of $k$ points for any $3 \le k <\infty$. To this end, we modify the graph $\Gamma$ so that instead of having $2$ branches, it has $k$ branches. The result follows by applying the same construction as above.

	
\end{example}
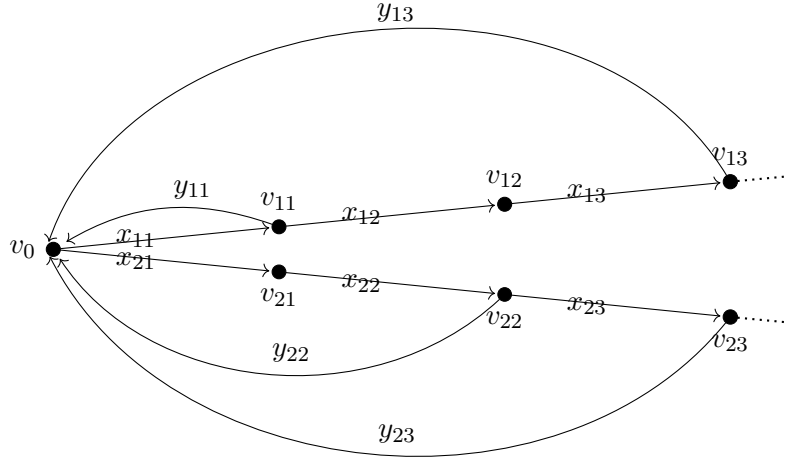
\begin{figure} [h!]
	\begin{tikzpicture}
	\node[shape=circle,draw=black,fill=black,scale=0.25,label=left:{$v_0$}] (A) at (-1,0) {A};
	\node[shape=circle,draw=black,fill=black,scale=0.25,label=above:{$v_{11}$}] (B) at (2,0.3) {B};
	\node[shape=circle,draw=black,fill=black,scale=0.25,label=above:{$v_{12}$}] (C) at (5,0.6) {C};
	\node[shape=circle,draw=black,fill=black,scale=0.25,label=above:{$v_{13}$}] (D) at (8,0.9) {D};
	\node[shape=circle,draw=black,fill=black,scale=0.25,label=below:{$v_{21}$}] (E) at (2,-0.3) {E};
	\node[shape=circle,draw=black,fill=black,scale=0.25,label=below:{$v_{22}$}] (F) at (5,-0.6) {F} ;
	\node[shape=circle,draw=black,fill=black,scale=0.25,label=below:{$v_{23}$}] (G) at (8,-0.9) {G};
	\node[shape=circle,draw=none,fill=none,scale=0.25] (X) at (8.8,0.975) {};
	\node[shape=circle,draw=none,fill=none,scale=0.25] (Y) at (8.8,-0.975) {};
	
	\path [->, shorten  >=1pt] (A) edge node[left] {$x_{11}$} (B);
	\path [->, shorten  >=3pt] (B) edge[bend right=25] node[above right]  {$y_{11}$} (A.30);
	\path [->, shorten  >=1pt](B) edge node[left] {$x_{12}$} (C);
	\path [->, shorten  >=1pt](C) edge node[left] {$x_{13}$} (D);
	\path [->, shorten  >=3pt] (D) edge[bend right=65] node[above right]  {$y_{13}$} (A.180);
	\path [-] (D) edge[dotted,thick=0.9] (X);
	
	\path [->, shorten  >=1pt] (A) edge node[left] {$x_{21}$} (E);
	\path [->, shorten  >=1pt](E) edge node[left] {$x_{22}$} (F);
	\path [->, shorten  >=7pt] (F) edge[bend left=50] node[above right]  {$y_{22}$} (A.120);
	\path [->, shorten  >=1pt](F) edge node[left] {$x_{23}$} (G);
	\path [->, shorten  >=3pt] (G) edge[bend left=57] node[above right]  {$y_{23}$} (A.180); 
	\path [-] (G) edge[dotted,thick=0.9] (Y);
	\end{tikzpicture}
	\caption{\label{fig_ex} A variant of Peterson's example. Note that the edges $y_{ij}$ in the figure are only for illustration purposes. The first directed edge $y_{ij}$ going back from $v_{ij}$ to $v_0$ occurs for a large $j$.}
\end{figure}

\end{document}